\documentclass[12pt]{amsart}

\usepackage{amsmath} 
\usepackage{amsfonts}
\usepackage{amssymb}
\usepackage{amsthm}
\usepackage{mathtools}
\usepackage{latexsym}
\usepackage{color}
\usepackage{enumerate}
\usepackage{mathrsfs}
\usepackage{comment}
\usepackage{setspace}
\usepackage{soul}

\newtheorem{lemma}{Lemma}[section]
\newtheorem{thm}{Theorem}[section]
\newtheorem{prop}{Proposition}[section]
\newtheorem{cor}{Corollary}[section]

\theoremstyle{remark}
\newtheorem{remark}{Remark}[section]
\numberwithin{equation}{section}

\def\tr{\textmd{tr}}

\def\dint{\displaystyle\int}
\def\R{\mathbb{R}}

\def\R{\mathbb{R}}

\def\vh{\vspace{.1cm}}

\def\p{\partial}
\def\So{\S_0}

\def\m{\mathfrak{m}}
\def\mh{\mathfrak{m}_{_H}}
\def\a{\alpha}
\def\b{\beta}
\def\({\left(}
\def\){\right)}
\def\[{\left[}
\def\]{\right]}
\def\S{\Sigma}
\def\u{\overline u}
\def\me{m_e}
\def\mb{\mathfrak{m}_{_B}}

\newcounter{mnotecount}[section]

\begin{document}
	
	\title[Asymptotically Flat Extensions of CMC Bartnik Data]{Asymptotically Flat Extensions 
		\\of CMC Bartnik Data}
	
	\author[Cabrera]{Armando J. {Cabrera Pacheco}}
	\address{Department of Mathematics,  University of Connecticut, Storrs, CT 06269, USA.}
	\email{a.cabrera@uconn.edu}
	
	\author[Cederbaum]{Carla Cederbaum}
	\address{Department of Mathematics, Universit\"at T\"ubingen,  72076 T\"{u}bingen, Germany.}
	\email{cederbaum@math.uni-tuebingen.de}
	
	\author[McCormick]{Stephen McCormick}
	\address{Institutionen f\"{o}r Matematik, Kungliga Tekniska h\"{o}gskolan, 100 44 Stockholm, Sweden; and School of Science and Technology, University of New England, Armidale, NSW 2351, Australia.}
	\email{stephen.mccormick@une.edu.au}
	
	\author[Miao]{Pengzi Miao}
	\address{Department of Mathematics, University of Miami, Coral Gables, FL 33146, USA.}
	\email{pengzim@math.miami.edu}
	
	\begin{abstract}
		Let $g$ be a metric on the $2$-sphere $\mathbb{S}^2$ with positive Gaussian curvature and $H$ be a positive constant. Under suitable conditions on $(g, H)$, we construct smooth, asymptotically flat $3$-manifolds $M$ with non-negative scalar curvature, with outer-minimizing boundary isometric to $(\mathbb{S}^2, g)$ and having mean curvature $H$, such that near infinity $M$ is isometric to a spatial Schwarzschild manifold whose mass {$m$} can be made arbitrarily close to a constant multiple of the Hawking mass of $(\mathbb{S}^2,g,H)$. Moreover, this constant multiplicative factor depends only on $(g, H)$ and tends to $1$ as $H$ tends to $0$. 
		The result provides a new upper bound of the Bartnik mass associated to such boundary data. 
	\end{abstract}
	
	\maketitle 
	
	\section{introduction}
	Let $g$ be a metric on the $2$-sphere $\mathbb{S}^2$ for which the first eigenvalue of $ - \Delta + K $ is positive, where $ K$ is the Gaussian curvature of $g$. Mantoulidis and Schoen \cite{M-S} constructed asymptotically flat $3$-manifolds with non-negative scalar curvature, whose boundaries are outermost minimal surfaces that are isometric to $(\mathbb{S}^2, g)$ and with ADM mass $m_{ADM}$ \cite{ADM} that can be made arbitrarily close to the optimal value determined by the Riemannian Penrose inequality \cite{Bray01, H-I01}.
	
	In terms of the quasi-local mass $ \mb$ introduced by Bartnik \cite{Bartnik,BartnikTsingHua}, the result of Mantoulidis and Schoen can be reformulated  as follows:
	Given smooth \emph{Bartnik data} $(\S , g, H)$, i.e., a triple with $\S  \simeq \mathbb{S}^2$,  $g$ a metric and $H$ a function on  $ \S$, if $H=0$ and $g$ {has positive first eigenvalue}  $\lambda_1 ( - \Delta + K ) > 0 $, then {$\mb(\S,g,H)$ of such Bartnik data is bounded above by their Hawking mass $\mh(\S,g,H)$:}
	\begin{equation} \label{eq-M-S}
	\mb (\S, g, 0) \le \mh (\S , g, 0),
	\end{equation}
	where the Hawking mass $\mh$ \cite{Hawking} is defined  by
	\begin{align*}
	\mh ( \S, g , H) \coloneqq \sqrt{ \frac{ | \S |_{g} }{ 16 \pi} } \left( 1 - \frac{1}{16 \pi} \int_{\S} H^2 d \sigma \right).
	\end{align*}
	Bartnik's quasi-local mass $\mb$ is defined as
	\begin{align*}
	\mb(\S,g,H) &\coloneqq\inf\left\lbrace m_{ADM}(M,\gamma)\,\vert\,(M,\gamma) \text{ admissible extension of }(\S,g,H)\right\rbrace,
	\end{align*}
	where a smooth, asymptotically flat Riemannian $3$-manifold $(M,\gamma)$ with boundary $\p M$ is called an \emph{admissible extension of $(\S,g,H)$} if it has non-negative scalar curvature and if $(\S,g)$ is isometric to $\p M$ with mean curvature $H$. Moreover, it is required that $(M,\gamma)$ contains no closed minimal surfaces (except possibly $\p M$) or, a fortiori, that $\p M$ is outer-minimizing in $(M,\gamma)$,  see \cite{Bray01,Bray04,H-I01}. Via the proof of the Riemannian Penrose inequality in \cite{H-I01}, the outer-minimizing condition allows one to estimate the Bartnik mass of given Bartnik data from below by its Hawking mass.
	
	For the purpose of this paper, it will make no difference which of the 
	above two conditions -- no closed minimal surfaces versus outer-minimizing -- is chosen as 
	the extensions we construct satisfy both conditions. 
	Here and in the following, we will abbreviate $ \mh (\S) \coloneqq \mh (\S, g, H)$ and $\mb(\S) \coloneqq \mb (\S, g, H)$ whenever the choice of $g$ and $H$ is clear from context.
	
	In the important
	case of $H =0 $ treated by Mantoulidis and Schoen, combined with the Riemannian Penrose inequality, 
	\eqref{eq-M-S} implies $ \mb (\S, g, 0)  = \mh (\S, g, 0) $. We note that the condition $\lambda_1( - \Delta + K) > 0 $
	used in this setting arises naturally in the context of stable minimal surfaces. 
	
	In this paper, we give an analogue of Mantoulidis and Schoen's result  for a triple of Bartnik data $(\S, g, H)$
	that is associated to a constant mean curvature (CMC) surface. As a special case of our main Theorem (Theorem \ref{thm-extension}), we have  the following  result on constructing asymptotically flat extensions {with controlled ADM mass}. 
	
	\begin{thm}\label{thm-intro}
		Let $(\S\simeq\mathbb{S}^{2}, g, H )$ be a triple of Bartnik data where $H = H_o $ is a positive constant and $g$ 
		has positive Gaussian curvature. 
		There exist constants $ \alpha \ge 0  $ and  $ 0 < \beta \le 1 $, depending on $g$,  such that if 
		\begin{align}\label{eq-Willmore-condition}
		\frac{1}{16 \pi} \int_{\S} H_{o}^2 d \sigma < \frac{ \beta}{ 1 + \alpha } , 
		\end{align}
		then for each $m > m_*$, with
		\begin{align*}
		m_* =     \left[ 1 + \left(  \frac{ \alpha \left(\frac{1}{16\pi}  \int_{\S} H_{o}^2 d \sigma  \right) } { \beta  - \left( 1 +  \alpha \right) \left( \frac{1}{16\pi}\int_{\S} H_{o}^2 d \sigma   \right) } \right)^\frac12  \right]  \mh (\S,g,H_o),
		\end{align*}
		there exists a smooth, asymptotically flat {Riemannian }$3$-manifold $(M,\gamma)$ with boundary $\p M$ and non-negative scalar curvature such that
		\begin{enumerate}[(i)]
			\item $\p M$ is isometric to $(\S,g)$ and has constant mean curvature $H_o$; 
			\item $M$, outside a compact set, is isometric to a spatial Schwarzschild manifold $M^S_m$ of mass $m$; and 
			\item $M$ is foliated by mean convex $2$-spheres which  eventually coincide with the rotationally symmetric $2$-spheres  
			in $M^S_m$. 
		\end{enumerate}
	\end{thm}
	
	\begin{remark}
		Condition \eqref{eq-Willmore-condition} implies 
		\begin{align*}
		\frac{1}{16\pi}\int_\S H_{o}^2 d \sigma < \frac{\beta}{1 + \alpha} \le 1,
		\end{align*}
		hence $ \mh (\S) > 0 $ is implicitly assumed in Theorem \ref{thm-intro}.
		Moreover, \emph{(iii)} implies that $\p M$ is outer-minimizing in $M$.
	\end{remark}
	
	Similar to the implication of Mantoulidis and Schoen's result on the Bartnik mass, Theorem \ref{thm-intro} has the following direct corollary. 
	
	\begin{cor}
		Let $(\S\simeq\mathbb{S}^{2}, g, H )$ be a triple of Bartnik data where $H = H_o $ is a positive constant and $g$ 
		has positive Gaussian curvature.  Suppose  \eqref{eq-Willmore-condition} holds, then 
		\begin{align}\label{eq-B-bound}
		\mb (\S, g, H_o) \le  \left[ 1 + \left(  \frac{ \alpha \left(\frac{1}{16\pi}  \int_{\S} H_{o}^2 d \sigma  \right) } { \beta  - \left( 1 +  \alpha \right) \left( \frac{1}{16\pi}\int_{\S} H_{o}^2 d \sigma   \right) } \right)^\frac12  \right]  \mh (\S,g,H_o)  .
		\end{align}
	\end{cor}

	\begin{remark}
		Inequality \eqref{eq-B-bound} has the feature that the ratio between the upper bound it provides for 
		$\mb(\S)$ and the Hawking mass $\mh (\S)$ tends to $1$ as $H_o \to 0$. However, it assumes 
		\eqref{eq-Willmore-condition}. In \cite{L-S}, Lin and Sormani investigated the Bartnik mass of arbitrary 
		CMC Bartnik data $(\S, g, H_o)$, with area $|\S|_g=4\pi$, and proved that
		\begin{align}\label{eq-L-S}
		\mb (\S) \le  \m_{a\mathbb{S}} (\S) +  \mh (\S) ,
		\end{align}
		where $\m_{a\mathbb{S}} (\S)$, referred to as the \emph{asphericity mass}, is a non-negative constant 
		that is determined only by the metric $g$ on $ \S$. Note that, in contrast to 
		\eqref{eq-B-bound}, the ratio between the 
		right hand side of \eqref{eq-L-S} and the Hawking mass approaches a fixed constant $C> 1$ as $H \to 0$.
	\end{remark}
	
	\begin{remark}
		The quantity $ m_*$ in Theorem \ref{thm-intro} was  first found 
		by Miao and Xie  in \cite{M-X}. Indeed, the Bartnik mass 
		estimate \eqref{eq-B-bound} would follow from the result in \cite{M-X} if one allows admissible extensions 
		to be Lipschitz  with distributional non-negative scalar curvature across  a hypersurface (cf. \cite{Miao02, ShiTam02}). 
		Therefore, our main contribution here is to construct smooth extensions. 
	\end{remark}

	We will prove Theorem \ref{thm-intro} applying the method of  Mantoulidis and Schoen \cite{M-S}, which involves two steps. In the first step, one needs a collar extension of $(\S, g, H_o)$ with positive scalar curvature on $ [0, 1 ] \times \S$ such that the geometric information at $\S_0\coloneqq\{ 0 \} \times \S$, corresponding to $ (\S, g, H_o)$, is suitably propagated to the other end $\S_1 \coloneqq \{ 1 \} \times \S$, near which the extension is rotationally symmetric. (In the minimal surface case, it is primarily the area of $\S_1$ that one wants to compare with $\S_0$. When $H_o > 0 $, it is the two Hawking masses that one wants to compare.) In the second step, one smoothly glues the collar extension, at $ \S_1$, to a spatial Schwarzschild manifold (suitably deformed in a small region) with mass greater than, but arbitrarily close to, the Hawking mass of $ \S_1$.
	
	To implement this process, we make use of the collar extension constructed in \cite{M-X} in the first step as it provides a good control of the Hawking mass along the collar. For the second step, we prove  
	an elementary  result on smoothly gluing 
	a  rotationally symmetric manifold with positive scalar curvature and 
	a spatial Schwarzschild manifold with suitably chosen mass  (see Proposition \ref{prop-extension}). 
	We give these gluing tools in Section \ref{sec-gluing}, then in Section \ref{sec-main-theorem}, 
	we prove Theorem \ref{thm-extension} which implies Theorem \ref{thm-intro}. 
	
	\vspace{.3cm}
	
	\noindent {\em Acknowledgements.}  
	The work of CC and SM was partially supported by the DAAD and Universities Australia. CC is indebted to the Baden-W\"urttemberg Stiftung for the financial support of this research project by the Eliteprogramme for Postdocs. The work of CC is supported by the Institutional Strategy of the University of T\"ubingen (Deutsche Forschungsgemeinschaft, ZUK 63).
	The work of PM was partially supported by Simons Foundation Collaboration Grant for Mathematicians \#281105.

	\section{Gluing to a spatial Schwarzschild manifold}  \label{sec-gluing}
	In this section, we list some tools for gluing together two rotationally symmetric metrics with non-negative scalar curvature, 
	which are used to prove the main result  in Section \ref{sec-main-theorem}. 
	We start with a lemma  that  is  a slight generalization of \cite[Lemma 2.2]{M-S}. 
	As it may be of independent interest, we state it for arbitrary dimension $n \geq 2$.
	
	\begin{lemma} \label{gluing-lemma}
		Let $f_i:[a_i,b_i]\to \R{^{+}}$, where $i=1,2$, be smooth positive functions, and let $g_*$ be the standard metric on $\mathbb{S}^n$. Suppose that
		\begin{itemize}
			\item[(i)]  the metrics $\gamma_{i}\coloneqq dt^2+f_i(t)^2g_*$ have positive scalar curvature; \label{itemscalar}
			\item[(ii)] $f_1(b_1)<f_2(a_2)$; 
			\item[(iii)]    $ 1 > f_1'(b_1) > 0 $ and $ f_1'(b_1)  \geq f_2'(a_2) >  -1  $. 
		\end{itemize}
		Then,  after translating the intervals so that 
		\begin{equation} \label{eq-translation}
		\left\{ 
		\begin{array}{r l}
		( a_2  - b_1  )  f_1'(b_1)    =     f_2 (a_2) - f_1 (b_1), & \ \mathrm{if} \   f_1'(b_1)  = f_2'(a_2) , \\
		( a_2  - b_1  )  f_1'(b_1)   >   f_2 (a_2) - f_1 (b_1)  >  (a_2 - b_1) f_2'(a_2)   ,  & \ \mathrm{if} \ f_1'(b_1)  > f_2'(a_2) ,
		\end{array}
		\right.
		\end{equation}
		one can construct a smooth positive function $f:[a_1,b_2]\to\R{^{+}}$ so that:
		\begin{itemize}
			\item [(I)] $f\equiv f_1$ on $[a_1,\frac{a_1+b_1}{2}]$, $f\equiv f_2$ on $[\frac{a_2+b_2}{2},b_2]$, and
			\item[(II)]  $\gamma\coloneqq dt^2+f(t)^2g_*$ has positive scalar curvature on $[a_1,b_2] \times \mathbb{S}^n$.
		\end{itemize}
		Moreover, if $ f_i' > 0 $ on $[a_i, b_i]$, then $ f $ can be constructed {such} that $ f' > 0 $ on $[a_1, b_2]$.
	\end{lemma}
	
	\begin{remark}
		The assumptions in Lemma \ref{gluing-lemma} are weaker than those  in  \cite[Lemma 2.2]{M-S}
		employed by Mantoulidis and Schoen.
		In \cite[Lemma 2.2]{M-S}, it was assumed {that} $ f_i' > 0$ and $ f_i'' > 0 $, which implies that $ f_i' < 1 $ 
		since each $\gamma_{i}$ has positive scalar curvature (cf. \eqref{Rfcond}); moreover,  we relax  
		the assumption $f'_1(b_1)=f'_2(a_2)$ in \cite[Lemma 2.2]{M-S}   to $f'_1(b_1) \geq f'_2(a_2)$, 
		which corresponds to a mean curvature jump condition used in the proof of the positive mass 
		theorem with corners by Miao  \cite{Miao02}.
		This  assumption in terms of an inequality helps one 
		simplify  the  metric gluing procedure used later.
	\end{remark}
	
	\begin{proof}
		By (ii) and (iii), the interval $[a_2, b_2]$ can always be translated so that  \eqref{eq-translation} holds. 
		Assume that such a translation has been performed.
		On  $[b_1, a_2]$, \eqref{eq-translation} implies that there exists 
		a function $\zeta\in C^1([b_1, a_2])$ such that $\zeta (b_1) = f_1' (b_1)$, $ \zeta (a_2) = f_2' (a_2)$, $\zeta ' \le 0$, and
		$  \int_{b_1}^{a_2} \zeta (t)\, d t = f_2 (a_2) - f_1 (b_1)  $.
		On  $[b_1, a_2]$, define 
		$$ \widehat f (t) \coloneqq f_1(b_1) + \int_{b_1}^t  \zeta(x) \,d x .$$
		Then $\widehat f$ satisfies 
		\begin{align*}
		\widehat f (b_1) &= f_1(b_1) \text{ and } \widehat f (a_2)= f_2(a_2),\\
		{\widehat f} ' (b_1) &= f_1'(b_1) \text{ and } \widehat f\,' (a_2) = f_2'(a_2),\\
		1 > f_1'(b_1 ) &\geq \widehat f'(t) \geq f_2'(a_2) > -1  \text{ on }(b_1,a_2),\\
		{\widehat f} \,''(t) &=  \zeta'(t) \le 0  \text{ on }[b_1,a_2].
		\end{align*}
		On $[a_1, b_2]$, define 
		\begin{align*}
		\widetilde{f}(t) \coloneqq\left\{\begin{matrix}
		f_1(t)&\text{ on } [a_1,b_1]\\
		\widehat f (t)&\text{ on } [b_1,a_2]\\
		f_2(t)&\text{ on } [a_2,b_2]
		\end{matrix}\right. . 
		\end{align*}
		Then	$\widetilde{f}\in C^{1,1}([a_1, b_2])$, $C^{2}$ away from $b_{1}$, $a_{2}$, and $\widetilde f > 0 $. Slightly abusing notation, we will write $\widetilde{f}''$ on all of $[a_{1},b_{2}]$ extending it to $b_{1}$, $a_{2}$ by setting $\widetilde f''(b_{1})\coloneqq f_{1}''(b_{1})$ and $\widetilde f''(a_{2})\coloneqq f_{2}''(a_{2})$. 
		
		We next consider an appropriate mollification of $ \widetilde{f}$ (cf.  \cite{C-M}). 
		Let $\delta >0$ be such that
		\begin{align*}
		\dfrac{a_1+b_1}{2} < b_1 - \delta\ \ \text{and} \ \ a_2+\delta < \dfrac{a_2+b_2}{2}. 
		\end{align*}
		Let  $\eta_{\delta}:[a_{1},b_{2}]\to\mathbb{R}^{+}_{0}$ be a smooth cutoff function which equals $1$ on $[b_1-\delta,a_2+\delta]$, vanishes  on $\left[a_1,\frac{a_1+b_1}{2} \right]\cup \left[\frac{a_2+b_2}{2},b_2\right]$, and satisfies $ 0 <  \eta_{\delta} (s)  < 1 $ in the remaining part of the interval. Let $\phi:\mathbb{R}\to\mathbb{R}^{+}_{0}$ be a standard smooth mollifier with compact support in $[-1,1]$ and $\int_{-\infty}^\infty\phi(s)\,ds=1$.
		
		For each positive $\varepsilon <  \frac{\delta}{4}  $, define $f_{\varepsilon}$ by 
		\begin{align}\label{eq-def-fe}
		f_{\varepsilon}(t)\coloneqq\dint_{-\infty}^\infty \widetilde{f}(t-\varepsilon \eta_{\delta}(t)s)\phi(s)\, ds , \ \  t \in [a_1, b_2]  
		\end{align}
		and observe that $f_{\varepsilon}$ is smooth on $[a_1,b_2]$. Then $f_\varepsilon\equiv \widetilde{f}$ on $\left[a_1,\frac{a_1+b_1}{2} \right]\cup \left[\frac{a_2+b_2}{2},b_2\right]$ and 
		\begin{align}\label{eq-def-fe'}
		f'_{\varepsilon}(t)=\dint_{-\infty}^\infty \widetilde{f}' (t-\varepsilon \eta_{\delta}(t)s)(1 - \varepsilon \eta_\delta'(t) s  ) \phi(s)  \,  ds  \ \forall\, t \in [a_1, b_2] .
		\end{align}
		Moreover, since $ \widetilde f'$ is $C^0$ everywhere and $C^1$ except at $b_1$ and $a_2$, it can be checked that by standard mollification arguments
		\begin{align}\label{eq-def-fe''-1}
		\begin{split}
		f''_{\varepsilon}(t) = & \  \frac{d}{dt} \left(  \dint_{-\infty}^\infty \widetilde{f}' (t-\varepsilon s) \phi(s)  \,  ds , \right) \\
		= & \   \dint_{-\infty}^\infty \widetilde{f}''  (s) \phi_\varepsilon ( t -s)   \,  ds  \ \ \forall\,t \in (b_{1}-\delta,a_{2}+\delta), 
		\end{split}
		\end{align}
		where $ \phi_\varepsilon(t) \coloneqq \frac{1}{\varepsilon} \phi ( \frac{ t}{\varepsilon} )$, and
		\begin{align}\label{eq-def-fe''-2} 
		\begin{split}
		f''_{\varepsilon}(t) = & \ \int_{-\infty}^\infty \widetilde{f}'' (t-\varepsilon \eta_{\delta}(t)s)(1 - \varepsilon \eta_\delta'(t) s  )^2 \phi(s)  \,  ds  \\ 
		& \ - \varepsilon \int_{-\infty}^\infty \widetilde{f}' (t-\varepsilon \eta_{\delta}(t)s)  \eta_\delta'' (t) s   \phi(s)  \,  ds, \ 
		\forall\,t \notin [b_1 -  \frac14 {\delta}, a_2 + \frac14  \delta] . 
		\end{split} 
		\end{align}
		
		We claim that, for sufficiently small $\varepsilon>0$,  the metric $d t^2 + f_\varepsilon (t)^2  g_*$ has positive scalar curvature. 
		To shows this, recall that given any smooth  function $f > 0$, the metric $ dt^2+f(t)^2g_*$ has positive scalar curvature if and only if 
		\begin{align}\label{Rfcond}
		f''(t)<\frac{n-1}{2f(t)}(1-f'(t)^2)\quad\forall\,t.
		\end{align}
		Suggested  by \eqref{Rfcond}, given any positive  $f \in C^1([a_1, b_2])$, we define 
		\begin{align*}
		\Omega[f](t) : =\frac{n-1}{2 {f}(t)}(1- {f}'(t)^2), \quad t\in[a_1,b_2].
		\end{align*}
		Observe that $\Omega[f]\in C^0 ([a_1, b_2])$. Then, on $ [a_1, b_2] \setminus \{ b_1, a_2 \} $,   we have 
		\begin{align*}
		\Omega [ \widetilde f\, ]  > \widetilde f'' 
		\end{align*}
		because $ \Omega [ \widetilde f \,]  > 0 $ on $[b_1, a_2]$, $\widetilde f'' \le 0 $  on $(b_1, a_2)$, and $ g_i$ has positive scalar curvature on $[a_i, b_i] \times \mathbb{S}^{n}$ for $i=1,2$. Moreover, for the same reason, we indeed have
		\begin{align}\label{eq-inf-p}
		3d\coloneqq\inf_{ t \in  [a_1, b_2] \setminus \{ b_1, a_2 \} } \left( \Omega [ \widetilde f\, ]  - \widetilde f'' \right) > 0 .
		\end{align}
		Thus,
		\begin{align*}
		\widetilde{f}'' (t) \leq\Omega[\widetilde{f}\, ] (t) -3d
		\end{align*}
		wherever $\widetilde f''$ is defined. 
		Now we consider $ \Omega [ f_\varepsilon ]$. By \eqref{eq-def-fe} and \eqref{eq-def-fe'},  $ f_\varepsilon \to \widetilde f $ in $ C^1([a_1, b_2])$ as $\varepsilon\to0^{+}$. Hence, $ \Omega[f_\varepsilon] \to \Omega[\widetilde{f}\,]$ in $ C^0 ([ a_1, b_2])$, 
		which shows, for small $ \varepsilon$, 
		\begin{align*}
		\sup_{ t \in [a_1, b_2] } \left| \Omega[\widetilde{f}\,] (t) -\Omega[f_\varepsilon] (t) \right|<d.
		\end{align*} 
		Moreover, by \eqref{eq-def-fe''-1} and \eqref{eq-def-fe''-2}, 
		\begin{align*}
		f''_\varepsilon(t)<\sup_{ | s - t | < \varepsilon }\widetilde f''(s)+d \quad \forall\,t \in[a_{1},b_{2}]
		\end{align*}
		provided $ \varepsilon $ is sufficiently small. 
		Therefore, it follows that, for all $t\in[a_{1},b_{2}]$,
		\begin{align*}
		f''_\varepsilon(t)
		&\leq \sup_{ | s - t | < \varepsilon } \left( \Omega[\widetilde{f}\,](s)-3d \right) +d\\
		&<\Omega[\widetilde{f}\,](t)-d\\
		&<\Omega[f_\varepsilon](t),
		\end{align*}
		where, in the second to last inequality, we  also used the fact that $\Omega[\widetilde{f}\,]$ is uniformly continuous on $[a_1, b_2]$, and hence $ \Omega[\widetilde{f}\,](s)  < \Omega[\widetilde f\,] (t) + d $ for any $ s $ with $ | s - t | < \varepsilon $  provided $ \varepsilon $ is small.
		
		Thus, we have shown that \eqref{Rfcond} holds on $[a_1, b_2]$ with $f$ replaced by $ f_\varepsilon$ for small $\varepsilon $. 
		Hence, the metric $ d t^2 + f_\varepsilon(t) ^2 g_*$ has positive scalar curvature and thus $f\coloneqq f_{\varepsilon}$ satisfies the conclusions of the theorem for small enough $\varepsilon$. Finally, if it is assumed that $ f_i' > 0 $, then $ \widetilde f' > 0 $ which directly implies $ f'=f_\varepsilon' > 0 $ as $ f_\varepsilon \to \widetilde f $ in $C^1([a_1, b_2])$.
	\end{proof}
	
	\begin{remark} \label{rem-g-lemma}
		It follows  from the above proof that, if $ f'_1 (b_1) > f'_2 (a_2)$, the assumptions $  f_1' (b_1) < 1 $ and $ f_2'(a_2) > -1 $  can be weakened to $ f_1' (b_1) \le 1 $ and $ f_2'(a_2) \ge -1$, respectively. This is because, in this case, one can require $ \zeta'(t) < 0 $ on $[b_1, a_2]$. Thus, $ \widehat f '' \le  \max_{[b_1, a_2]} \zeta' < 0   $ on $ (b_1, a_2)$ while $ \Omega [ \widetilde f] \ge 0 $ on $[b_1, a_2]$.  Therefore, \eqref{eq-inf-p} still holds.
	\end{remark}
	
	\begin{prop}\label{prop-extension}
		Consider a metric $ \gamma= ds^2 + f(s)^2 g_* $ on $ [a, b] \times \mathbb{S}^2$, where $ g_*$ is the standard metric on $ \mathbb{S}^2$ and $f > 0 $ is a smooth function on $[a, b]$. Suppose 
		\begin{enumerate}
			\item $\gamma $ has positive scalar curvature;
			\item $ \S_b  : = \{ b \} \times \mathbb{S}^2$ has positive mean curvature; and
			\item   $ \mh (\S_b  ) \ge  0 $. 
		\end{enumerate}
		Then, for any $\me > \mh(\S_b)$, there exists a smooth, rotationally symmetric, asymptotically flat Riemannian $3$-manifold $M$ with boundary $\p M$ and with non-negative scalar curvature such that 
		\begin{enumerate}[(i)]
			\item $M$, outside a compact set, is isometric to a spatial Schwarzschild manifold of mass $m_e$;
			\item  $ \p M $ has a neighborhood $U$  that is 
			isometric to $ \left( \left[ a, \frac{a+b}{2} \right) \times \mathbb{S}^2, \gamma \right)$; and
			\item if  $ f' > 0 $ on $[a, b]$, then $M$ can be constructed such that every rotationally symmetric sphere in $M$ has positive constant mean curvature. 
		\end{enumerate} 
	\end{prop}
	
	To prove this proposition, we recall that a spatial Schwarzschild manifold $ \left(M^S_m, \gamma_m \right) $ with mass $m > 0$ takes the form of 
	\begin{align} 
	\left(M^S_m, \gamma_m \right)  = \left( (2m , \infty) \times \mathbb{S}^2, \frac{1}{1 - \frac{2 m}{r} } d r^2 + r^2 g_*  \right).
	\end{align}
	Setting $s \coloneqq \int_{2m}^r \left( 1 - \frac{2m}{ t } \right)^{-\frac12} d t$, 
	we  can slightly extend $(M^{S}_{m},\gamma_{m})$ to  
	\begin{align}  \label{eq-S-metric} 
	\left(M^S_m, \gamma_m \right)  =  \left( [0, \infty) \times \mathbb{S}^2, ds^2 + \u_m^2 (s) g_* \right) , 
	\end{align}
	where $\u_m:[0,\infty)\to[2m,\infty)$ satisfies $\u_{m}(0)=2m$,
	\begin{align}\label{eq-lum-p}
	\u_{m}'(s)=\sqrt{1-\dfrac{2 m}{\u_{m}(s)}} \ \ 
	\mathrm{and}
	\ \ 
	\u_{m}''(s)=\dfrac{m}{\u_{m}(s)^2} . 
	\end{align}
	
	\begin{proof}[Proof of Proposition \ref{prop-extension}]
		The Hawking mass of $\S_b$ is given by
		\begin{align}\label{hawking-1}
		\mh(\S_b)  = \dfrac{f (b) }{2}\( 1- f '(b)^2 \).     
		\end{align}
		For simplicity, 
		we {set} $m_*\coloneqq\mh(\S_b)$, $r_*\coloneqq f (b)$ and $w_* \coloneqq f '(b)$, so \eqref{hawking-1} becomes
		$ m_*=\dfrac{r_*}{2}(1-w_*^2) $.
		By assumption, $ w_* > 0 $ and $ m_* \ge 0 $. The latter  implies $ w_*  \le   1 $. 
		
		Given any $\me>m_*$, we claim that there there exists $s_{\me}>0$ such that 
		\begin{align}\label{eq-ume-l}
		\u_{\me}'(s_{\me}) &\le w_*, \\\label{eq-umep-g}
		\u_{\me}(s_{\me})&>  r_* . 
		\end{align}
		In case $ m_* > 0 $, i.e., $ 0 < w_* < 1$, this is easily seen by choosing $ s_{\me}>0 $ so that 
		$ \u_{\me}' (s_{\me}) =  w_* $. 
		Then  \eqref{eq-umep-g} follows from \eqref{eq-lum-p} and the fact that $ m_e > m_*$.
		If $ m_* = 0 $, i.e., $ w_* = 1 $, then \eqref{eq-ume-l} holds (with a strict inequality) for any $ s_{\me} > 0 $. In this case, \eqref{eq-umep-g} holds for any sufficiently large  $ s_{\me} $.
		
		To proceed, we fix $ s_{\me}$. Let $ \delta \in (0, s_{\me}) $ be a small  constant to be chosen later. 
		We now slightly bend  a small piece of the Schwarzschild manifold 
		$\left(M^S_{\me}, \gamma_{\me} \right)$,
		near $s=s_{m_e}$, so that it has positive scalar curvature and can be smoothly glued here to 
		$\left( \left[a, b \right]  \times \mathbb{S}^2 , \gamma \right) $ using Lemma \ref{gluing-lemma}.
		
		Following Lemma 2.3 in \cite{M-S}, we consider a smooth function $\sigma:[s_{m_e} - \delta, \infty)\to\mathbb{R}$ defined by $\sigma (s) \coloneqq s$ for $s \ge s_{m_e} $, and indirectly via
		\begin{align*}
		\sigma'(s)\coloneqq1+e^{-1/(s-s_{\me})^2}  \text{ for } s_{m_e} - \delta \le s \le s_{m_e} . 
		\end{align*}
		If $ \delta $ is small enough then 
		$\sigma (s) > 0 $  and the metric
		\begin{align*}
		\gamma_{e} \coloneqq  ds^2 + \u_{m_e} ( \sigma (s) )^2 g_*
		\end{align*}  
		has positive scalar curvature  on $[s_{\me}-\delta,s_{\me}) \times \mathbb{S}^2$. Note that $\gamma_e$  is identically the Schwarzschild metric with mass $m_e$ on $[s_{m_e} , \infty) \times \mathbb{S}^2 $ and $ \frac{d}{d s}  \u_{m_e} ( \sigma (s) ) > 0 $,  $\forall\,s \ge s_{\me} - \delta $. 
		
		We now want to compare $ \u_{m_e}  ( \sigma (s) ) |_{s =  s_{m_e} - \delta } $ and $ \frac{d}{ds} \u_{m_e}  ( \sigma (s) ) |_{s =  s_{m_e} - \delta } $, with $r_*$ and $ w_*$, respectively. By \eqref{eq-umep-g}, we have 
		\begin{align}\label{eq-prop-cond-1}
		\u_{m_e}  ( \sigma (  s ) )  |_{s =  s_{m_e} - \delta }   > r_* 
		\end{align}
		if $ \delta$ is sufficiently small. By \eqref{eq-lum-p}, one finds
		\begin{align*}
		\frac{d^2}{ds^2} \u_{\me}(\sigma(s))
		= & \  {m_e} \u_{m_e} (\sigma (s) )^{-2}    \sigma'(s)^2+\u_{\me}'(\sigma(s))\sigma''(s) . 
		\end{align*}
		Hence $\frac{d^2}{ds^2} \u_{\me}(\sigma(s)) |_{ s = s_{m_e}} = m_e \u_{m_e} ( s_{m_e} )^{-2}  > 0 $ since $ m_e > 0 $. Thus,  shrinking $\delta$ if necessary, we have $ \frac{d^2}{ds^2} \u_{\me}(\sigma(s))  > 0 $ on $ [s_{m_e} - \delta, s_{m_e}]$. Combined with \eqref{eq-ume-l}, this implies 
		\begin{align}\label{eq-prop-cond-2}
		\frac{d}{ds} \u_{\me}(\sigma(s))\bigg\vert_{s=s_{\me}-\delta} <  \u_{\me}'(s_{\me}) \le w_*   . 
		\end{align}
		
		Now, by \eqref{eq-prop-cond-1}, \eqref{eq-prop-cond-2} and as $\gamma_e$ has positive scalar curvature on $ [s_{m_e} - \delta, s_{m_e}) \times \mathbb{S}^2$, we can apply Lemma \ref{gluing-lemma} (and Remark \ref{rem-g-lemma}) to smoothly glue $ \gamma$ and $ \gamma_e$ together, with $f$ playing the role of $ f_1 $ on $[a_1, b_1]\coloneqq[a, b]$, and $ \u_{m_e}\circ \sigma$ playing the role of $f_2$ on $ [a_2, b_2]\coloneqq [ s_{m_e } - \delta , s_{m_e } - \frac{ \delta}{2}]$. The resulting manifold satisfies all properties we require.
	\end{proof}
	
	\begin{remark}
		We would like to comment on the size of the ``neck" that is needed in the above proof to connect  
		$\left( \left[a, b \right]  \times \mathbb{S}^2 , \gamma \right) $ to  $\left(M^S_{\me}, \gamma_{\me} \right)$. 
		Precisely, this means that we want to estimate $ l \coloneqq  a_2 - b_1 $ after having  set 
		$[a_1, b_1] = [a, b]$, $ [a_2, b_2] =  [ s_{m_e } - \delta , s_{m_e } - \frac{ \delta}{2}]$, 
		and  having  translated the intervals  so that  \eqref{eq-translation} holds.
		By \eqref{eq-translation} and \eqref{eq-prop-cond-2},  $l$ satisfies 
		\begin{equation*} 
		\frac{  \u_{\me}(\sigma(s_{\me}-\delta)) - r_* }{ \frac{d}{ds} \u_{\me}(\sigma(s)) \vert_{s=s_{\me}-\delta}   }
		> l >  \frac{  \u_{\me}(\sigma(s_{\me}-\delta)) - r_* }{w_*}. 
		\end{equation*}
		Suppose $ \mh(\S_b)  > 0 $, by construction 
		we have $  \u_{\me}' (s_{\me}) =  w_* $ and consequently 
		\begin{equation*}
		\begin{split}
		& \  \lim_{\delta \to 0}  \frac{  \u_{\me}(\sigma(s_{\me}-\delta)) - r_* }{ \frac{d}{ds} \u_{\me}(\sigma(s)) \vert_{s=s_{\me}-\delta}   } \\
		=  & \  \frac{  \u_{\me}( s_{\me} ) - r_* }{  \u'_{\me}( s_{\me}  ) }   =   \frac{  \u_{\me}( s_{\me} ) - r_*  }{w_*}  \\
		= & \  \lim_{\delta \to 0} \frac{  \u_{\me}(\sigma(s_{\me}-\delta)) - r_* }{w_*}.  
		\end{split}
		\end{equation*}
		Thus, by choosing $ \delta $  small,  we see that $l$ 
		is arbitrarily close to $ \displaystyle L \coloneqq  \frac{  \u_{\me}( s_{\me} ) - r_* }{w_*}$. 
		On the other hand, we have 
		\begin{equation*}
		L =  \frac{ \left(  \frac{\me}{m_*}   - 1 \right) r_* }{w_*} ,
		\end{equation*} 
		which follows from the fact   $ m_*=\dfrac{r_*}{2}(1-w_*^2)  $,
		$ \me = \frac12  \u_{\me}( s_{\me} )  \left( 1 -  \u'_{\me}( s_{\me} )^2 \right)  $
		and  $  \u_{\me}' (s_{\me}) =  w_* $.
		Therefore, we conclude that $ l \to 0$ as $ \me \to \mh(\S_b) $. 
		
		When $ \mh(\S_b)  = 0 $, i.e.  $ w_* = 1$,  $l$ can always  be chosen to lie in 
		\begin{equation*} 
		\left( \frac{ \u_{\me}( s_{\me} ) - r_*  }{  \u'_{\me}( s_{\me}  )  } ,    \u_{\me}( s_{\me} ) - r_* \right) . 
		\end{equation*}
		In this case, by choosing $ s_{\me} $ such that $   \u_{\me}( s_{\me} )   $ is  arbitrarily close to $ r_*$, 
		we see that $l$ can be taken to approach $0$ for any given $ \me > \mh(\S_b) $.
	\end{remark}

	\section{Asymptotically flat extensions} \label{sec-main-theorem}
	Throughout  this section, let $(\S\simeq\mathbb{S}^{2}, g, H)$ be a  triple of Bartnik data with $K(g) > 0 $  and  constant mean curvature $ H=H_o > 0 $. Let $ r_o $ be the area radius of $g$, i.e., $ | \S|_g = 4 \pi r_o^2 $. 
	We always let  $\{ g(t) \}_{0 \leq t \leq 1}$ be a smooth path of metrics  on $\Sigma$ such that 
	\begin{enumerate}[(i)]
		\item if $g$ is a round metric,  $\{ g(t) \}_{0 \leq t \leq 1}$ is  the constant path with  $g(t) = g$, 
		$\forall\,t\in[0,1]$;  and 
		\item if $g$ is not a round metric,   $\{ g(t) \}_{0 \leq t \leq 1}$   is a  path of metrics  with positive Gaussian curvature 
		satisfying 
		$g(0)=g$, $g(1)$ is  a round  metric,   and 
		$$ \tr_{g(t)} g'(t)=0 , \  \forall\,t\in[0,1] .$$  
		(Existence of such a  $\{ g(t) \}_{0 \leq t \leq 1}$ is given  by Mantoulidis and Schoen's proof of 
		\cite[Lemma 1.2]{M-S}.)
	\end{enumerate}
	As in \cite[Section 2]{M-X}, we  let $\alpha $ and $\beta$ be   two  constants 
	determined by such a path  via
	\begin{equation}
	\begin{split}\label{eq-def-alpha-beta}
	\a\coloneqq\dfrac{1}{4} \max \limits_{t\in[0,1]}  \max\limits_{\S}  |g'(t)|^2_{g(t)}, \\ 
	\b \coloneqq r_o^2   \min \limits_{t\in[0,1]}   \min\limits_{\S} K(g(t)).
	\end{split}
	\end{equation}
	Clearly, 
	\begin{enumerate}
		\item[(a)] $\beta = 1 $ and $\alpha = 0 $, if $g$ is a round metric; and 
		\item[(b)] $ 0 < \beta < 1  $ and $ \alpha > 0$, if $g$ is not a round metric. 
	\end{enumerate}
	In assertion (b), one uses the Gau{\ss}-Bonnet Theorem  and the fact $|\Sigma|_{g(t) } = 4\pi r_o^2 $, $\forall \ t$.
	
	The following theorem illustrates how the gluing tools in the previous section and the collar extension
	in \cite{M-X} are combined to produce asymptotically flat extensions with suitable CMC  Bartnik data.
	
	\begin{thm} \label{thm-extension}
		Let $(\S\simeq\mathbb{S}^{2}, g, H)$ be a triple of Bartnik data where  
		$g$ has positive Gaussian curvature and  $H = H_o $ is a positive constant. 
		Let $r_o$ be the area radius of $g$,  i.e., $\vert\S\vert_{g}=4\pi r_o^{2}$. 
		Let $\{ g(t) \}_{0 \leq t \leq 1}$ be a path of metrics given in (i) or (ii) above.
		Let $ \alpha \ge 0  $ and  $ 0 < \beta \le 1 $ be the constants 
		defined in  \eqref{eq-def-alpha-beta} for  this path $\{ g(t) \}_{0 \leq t \leq 1}$. 
		Suppose the condition 
		\begin{align}\label{eq-thm-W-cond}
		\dfrac{1}{4}H_o^2 r_o^2 < \dfrac{\beta}{1+\a} 
		\end{align}
		holds. 
		Given any $ m \in ( - \infty, \frac12 r_o) $ satisfying 
		\begin{align}\label{eq-thm-choice-m}
		\dfrac{1}{4} H_o^2 r_o^2  <  \dfrac{\b}{1+\a}\( 1 -\dfrac{2  m }{r_o} \) ,
		\end{align}
		let $k > 0 $ be the  constant given by 
		\begin{align}  \label{eq-thm-k-def}
		k\coloneqq\dfrac{H_o r_o}{2} \(1-\frac{2m}{r_o}\)^{-\frac{1}{2}} . 
		\end{align} 
		Define
		\begin{align} \label{eq-m-star}
		m_*  \coloneqq  
		\left\{ 
		\begin{array}{lc}
		\displaystyle 
		\frac12   r_o \left[  \frac{  \frac14 H_o^2 r_o^2    \alpha} {  \left( \beta -   \frac14 H_o^2 r_o^2   \alpha \right)  - k^2  } \right]^\frac12  
		(1 - k^2) +  \mh  (\S)  , &   \  \mathrm{if} \ m < 0  , \\
		\displaystyle  
		\frac12  r_o \left[  \frac{ \alpha k^2 } { \beta  - \left( 1 +  \alpha \right) k^2 } \right]^\frac12 ( 1 - k^2)  
		+ \mh  (\S)  ,  &  \  \mathrm{if} \ m \ge 0 ,
		\end{array}
		\right.
		\end{align}
		where $ \mh (\S) \coloneqq \mh (\S, g, H_o)$. Then, for any $ \me > m_*$, there exists a smooth, asymptotically flat Riemannian $3$-manifold $(M,\gamma)$ with boundary $\p M$ and non-negative scalar curvature such that
		\begin{enumerate}[(i)]
			\item $\p M$ is isometric to $(\S,g)$ and has constant mean curvature $H_o$;
			\item $M$, outside a compact set,   is isometric to a spatial Schwarzschild manifold $M^S_{\me}$ of mass $ \me $; and 
			\item $M$ is foliated by mean convex $2$-spheres which  eventually coincide with the rotationally symmetric $2$-spheres  
			in $M^S_{\me}$. 
		\end{enumerate}
	\end{thm}
	
	\begin{proof}
		We note that  \eqref{eq-thm-W-cond} implies 
		$$
		\frac{1}{16 \pi} \int_\S H_o^2 d \sigma < \frac{\beta}{1 + \alpha} \le 1 ; 
		$$
		that is  $ \mh (\S ) > 0 $ by hypothesis. 
		
		We first consider the case that $g$ is not a round metric, i.e. $ \alpha > 0$. 
		In this case,  we will first prove the theorem  under  the  additional assumption that 
		\begin{align} \label{eq-round-near-end}
		g (t) = g(1), \ \forall\,t  \in  [1-\theta, 1] 
		\end{align}
		for some $\theta \in (0, \frac13) $.  Such a  condition is imposed so that later  we can directly apply 
		Proposition \ref{prop-extension}. 
		
		Now we describe the collar extension produced in \cite[Proposition 2.1]{M-X}. 
		Given the constant  $m \in ( - \infty,  \frac12 r_o  )$, consider  part of the spatial Schwarzschild metric  
		\begin{align*}
		\gamma_m = \frac{1}{1 - \frac{2 m}{r} } d r^2 + r^2 g_* 
		\end{align*}
		defined on $[r_o , \infty) \times \mathbb{S}^2$. 
		By a change of variable $s \coloneqq \int_{r_o}^r \left( 1 - \frac{2m}{t} \right)^{-\frac12} d t$, 
		one can write
		\begin{align*}
		\gamma_m = ds^2 + u_m^2 (s) g_* ,
		\end{align*}
		where $ s \in [0, \infty)$ and  $ u_m  $ is the inverse function of $ s = s (r)$. 
		(Note that, in the case of $m \ge 0 $,  $u_m$ is related to $ \u_m $ in \eqref{eq-S-metric} by 
		$  u_m (t)  \coloneqq  \u_m ( s_0 + t) $,  where $ s_0 > 0 $ is given  by $ \u_m (s_0 ) = r_o $.)
		Next, by \eqref{eq-thm-W-cond} --   \eqref{eq-thm-k-def}, we have 
		\begin{align}
		\left\{ 
		\begin{array}{ll}
		\displaystyle 
		\beta - \left[ 1 + \left( 1 - \frac{2m}{r_o} \right) \alpha \right] k^2 > 0 ,  & \ \mathrm{if} \ m < 0 
		\\ 
		\ & \ \\ 
		\displaystyle 
		\beta - ( 1 + \alpha) k^2 >  0 , &  \ \mathrm{if} \ m \ge 0 . 
		\end{array}
		\right.
		\end{align}
		Therefore, we can define a  constant $A_o > 0 $ by 
		\begin{align*}
		A_o \coloneqq 
		\left\{ 
		\begin{array}{lc}
		\displaystyle 
		r_o \left[  \frac{ \alpha} { \beta - \left[ 1 + \left( 1 - \frac{2m}{r_o} \right) \alpha \right] k^2 } \right]^\frac12,  & 
		\ \mathrm{if} \ m < 0 \\ 
		\displaystyle 
		r_o \left[ \frac{ \alpha} { \beta - ( 1 + \alpha) k^2} \right]^\frac12 , &  \ \mathrm{if} \ m \ge 0 . 
		\end{array}
		\right.
		\end{align*}
		Applying Proposition 2.1 of \cite{M-X} (and the subsequent Remark 2.1), 
		we know  the metric 
		\begin{align*}
		\gamma\coloneqq A_o^2 dt^2+ \dfrac{u_m(A_o kt)^2}{r_o^2} g(t)
		\end{align*}
		defined on $N \coloneqq [0,1] \times \S$
		satisfies:
		\begin{enumerate}[(i)]
			\item $R(\gamma)>0$, where $ R(\gamma)$ is  the scalar curvature of $\gamma$; 
			\item the induced metric on $\S_0\coloneqq\{0\} \times \S$ by $\gamma$ is $g$;
			\item the mean curvature  $H(0)$ of $\S_0$ is  $H(0)=H_o$;
			\item $\S_t \coloneqq\{ t  \} \times \mathbb{S}^2$  has positive constant mean curvature $\forall\,t\in[0,1]$; and 
			\item the Hawking mass of $\S_1 =\{ 1 \} \times \S$ is given by 
			\begin{align*}
			\mh(\S_1)&=\dfrac{1}{2}[u_m(A_o k )-r_o ](1-k^2)+\mh(\S)  .
			\end{align*}
		\end{enumerate}
		Moreover, as in \cite[Section 3]{M-X},  one can estimate  $ \mh (\S_1)$  by 
		\begin{align} \label{eq-mh-est}
		\mh(\S_1) \le & \ 
		\left\{ 
		\begin{array}{ll}
		\frac{1}{4}  H_o r_o A_o (1-k^2)+\mh(\S)  & \ \mathrm{if} \ m < 0 \\
		\ & \ \\
		\frac{1}{2} A_o k (1-k^2)+\mh(\S)  & \ \mathrm{if} \ m \ge  0 .
		\end{array}
		\right. 
		\end{align}
		Here 
		\begin{align*}
		\frac{1}{4}  H_o r_o A_o = \frac12   r_o \left[  \frac{  \frac14 H_o^2 r_o^2    \alpha} {  \left( \beta -   \frac14 H_o^2 r_o^2   \alpha \right)  - k^2  } \right]^\frac12 \ \mathrm{if} \ m < 0 ,
		\end{align*}
		and
		\begin{align*}
		\frac{1}{2} A_o k   = \frac12  r_o \left[  \frac{ \alpha k^2 } { \beta  - \left( 1 +  \alpha \right) k^2 } \right]^\frac12 \ \mathrm{if} \ m \ge  0 . 
		\end{align*}
		In other words, \eqref{eq-mh-est} is simply to assert 
		$$ \mh (\Sigma_1) \le m_* . $$
		By \eqref{eq-thm-choice-m} and \eqref{eq-thm-k-def}, we have
		$ 
		k^2 < \frac{\beta}{1 + \alpha} < 1 .
		$ 
		Therefore, 
		\begin{align*}
		0 <  \mh(\S) <  \mh(\S_1) \le m_*. 
		\end{align*}
		Upon a change of variable $ s \coloneqq A_o t $,  $(N, \gamma)$ becomes 
		\begin{align*}
		\left( [ 0, A_o ] \times \S ,  ds^2 +  \frac{f(s)^2}{  r_o^{2} } g ( A_o^{-1} s )  \right),
		\end{align*}
		where $ f(s) =  u_m(ks)$ and $  f'(s) > 0 $. On $[ (1-\theta) A_o , A_o ] \times \S$, by \eqref{eq-round-near-end}, 
		\begin{align*}
		\gamma = ds^2 + f(s)^2 g_*
		\end{align*}
		for some fixed round metric $g_*$ of area $4 \pi$. The theorem now follows readily from Proposition \ref{prop-extension}. 
		
		Next, we show that the theorem still holds when condition \eqref{eq-round-near-end} is  removed. 
		The idea is to simply  approximate $ \{ g(t) \}_{0 \le t \le 1}$ by a path satisfying \eqref{eq-round-near-end}.
		Given any small $\theta \in (0,\frac13)$,  let $ \xi_\theta (t) \ge 0  $ be a smooth function on $[0,1]$ such that
		\begin{equation} \label{eq-xi-theta}
		\left\{
		\begin{split}
		\ & \xi_\theta(t) = \frac{t}{1 - 2 \theta}  ,  \ \forall \ t \in [0, 1 - 3 \theta], \\
		\ &  \xi_\theta (t) = 1,  \ \forall \ t \in [ 1 - \theta, 1], \\ 
		\ &  0 \le \xi_\theta'(t) \le \frac{1}{1 - 2 \theta}  \ \forall  \ t \in [0,1] .
		\end{split}
		\right.
		\end{equation}
		(For instance, such an  $\xi_\theta (t)$ can be obtained by a usual mollification of the piecewise smooth function 
		that equals $ \frac{t}{1- 2 \theta } $ on $[0, 1 - 2 \theta] $ and equals $1$ on $[1 - 2 \theta, 1]$.)
		Consider the reparametrized path 
		$$ g_\theta (t) \coloneqq  g ( \xi_\theta (t) ) , $$
		which satisfies $ g_\theta (t) = g (1) $ for $ t \in [1 -\theta, 1]$.
		Let $ \alpha_\theta$ and $ \beta_\theta$ be the corresponding constants defined in \eqref{eq-def-alpha-beta} with 
		$ g(t)$ replaced by  $ g_\theta (t) $. 
		Clearly, 
		\begin{equation} \label{eq-b-theta}
		\beta_\theta = \beta . 
		\end{equation}
		We claim  
		\begin{equation} \label{eq-a-limit}
		\lim_{\theta \to 0 } \alpha_\theta = \alpha .
		\end{equation}
		To see this,  note that 
		$$ | g_\theta ' (t) |^2_{ g_\theta (t)}  =  ( \xi_\theta' (t) )^2    | g' (\xi_\theta  (t) )  |^2_{g (\xi_\theta (t) ) } , \ \forall \ t \in [0,1] . $$
		Hence, by \eqref{eq-xi-theta},
		$$   | g_\theta ' (t) |^2_{ g_\theta (t)}  =   (1 - 2 \theta )^{-2}   | g' ( \xi_\theta (t) )  |^2_{g ( \xi_\theta(t) )  } , \ \forall t \in [0, 1 - 3 \theta]  ,$$
		which shows 
		\begin{equation} \label{eq-lower-a} 
		\begin{split}
		\alpha_\theta  \ge & \  (1 - 2 \theta )^{-2}  \,  \dfrac{1}{4} \max \limits_{ s \in \left[0, \frac{1- 3 \theta}{1 - 2 \theta}  \right]}  \max\limits_{\S}  |g '(s )|^2_{g (s )}  . 
		\end{split}
		\end{equation}
		On the other hand, also by \eqref{eq-xi-theta},
		\begin{equation}  \label{eq-upper-a}
		\alpha_\theta \le   (1 - 2 \theta )^{-2} \a . 
		\end{equation} 
		Letting $\theta \to 0$,  \eqref{eq-a-limit}  follows from \eqref{eq-lower-a} and \eqref{eq-upper-a}. 
		
		Now,  by  \eqref{eq-b-theta} and \eqref{eq-a-limit},  we may assume 
		$$ \frac14 H_o^2 r_o^2 <  \frac{\beta_\theta }{1 + \alpha_\theta}   \  \ 
		\mathrm{and} \  \ 
		\frac14 H_o^2 r_o^2 < \frac{\beta_\theta }{1 + \alpha_\theta}  \left( 1 - \frac{2m}{r_o} \right) $$
		for sufficiently  small $\theta$. 
		For these $\theta$,  Theorem \ref{thm-extension}  holds for every $m_e $ satisfying 
		\begin{equation} \label{eq-m-theta}
		m_e > m_{\theta *} , 
		\end{equation}
		where $m_{\theta *} $ is  the  constant in \eqref{eq-m-star} defined via $\alpha_\theta$ and $\beta_\theta$.
		Now suppose 
		$ m_e > m_{*} $. 
		Since
		$ \lim_{ \theta \to 0} m_{\theta *} = m_* $  by \eqref{eq-b-theta} and \eqref{eq-a-limit}, 
		$m_e $ must satisfy \eqref{eq-m-theta} for small $\theta$. Thus, 
		the conclusion of Theorem \ref{thm-extension}  holds for  $m_e $.  
		This completes the proof of Theorem \ref{thm-extension} 
		if $g$ is not a round metric. 
		
		Finally, we consider the case that $g$ is a round metric, say $ g = r_o^2 g_*$. In this case, we have $ \alpha = 0 $, $ m_* = \mh(\S) $, and $(\S, g, H_o)$ is simply  the boundary data of the spatial Schwarzschild metric $  \gamma_{m_*} = ds^2 + \u_{m_*} (s)^2 g_*  $ on $[s_0, \infty) \times \mathbb{S}^2 $, where   $ \u_{m_*} (s_0) = r_o $. Now pick any $ s_1 > s_0 $.  Let $ \delta > 0 $ be a small  constant to be chosen. Similar to  Lemma 2.3  in \cite{M-S}, we  consider  a smooth  function $ \tau$  on $ [s_0,  s_1 + \delta ]$ defined by  
		$\tau (s) = s $ for $ s \in  [ s_0, s_1]  $, and indirectly via
		\begin{align*} 
		\tau'(s)=1 -  e^{-1/(s-s_{1})^2} , \  s_1 \le s \le s_1 + \delta  . 
		\end{align*}
		Define a metric $ \gamma_{B} \coloneqq  ds^2 + \u_{m_*} ( \tau (s) )^2 g_* $. Then $\gamma_B$ coincides with $\gamma_{m_*}$ on $[s_0, s_1] \times \mathbb{S}^2$ and has positive scalar curvature  on $( s_1, s_{1} + \delta ] \times \mathbb{S}^2$ for sufficiently small $\delta$. Let $\mh (\S_{s_1} )$ and $\mh (\S_{s_1+ \delta} )$ denote the Hawking masses of $ \S_{s_1} = \{ s_1  \} \times \mathbb{S}^2  $ and $ \S_{s_1+ \delta} = \{ s_1 + \delta \} \times \mathbb{S}^2 $ with respect to $ \gamma_B$, respectively. Then $\mh (\S_{s_1} ) = m_*$, and  $\mh (\S_{s_1+ \delta} )$ can be made arbitrarily close to  $\mh (\S_{s_1} )$, provided $\delta$ is sufficiently small. 
		Thus, given any fixed $ m_e > m_* $,  we can choose $\delta  $ sufficiently small to ensure
		$$
		m_e  > \mh (\S_{s_1+ \delta} ) > 0 . 
		$$
		Moreover, we can assume $ \frac{d}{ds} \u_{m_*} (\tau (s) ) |_{s = s_1 + \delta} > 0 $, i.e., $ \S_{s_1+\delta} $ has positive mean curvature with respect to $\gamma_B$. The theorem now again follows from Proposition \ref{prop-extension}. 
	\end{proof}

	\begin{remark}
		Taking $m = 0 $ in  Theorem \ref{thm-extension}, one has  
		$ k = \frac12 H_o r_o$ and 
		\begin{align*}
		\dfrac{1}{2} A_o k (1-k^2)  =   \left[  \frac{ \alpha \left(  \frac14 H_o^2  r_o^2 \right) } { \beta  - \left( 1 +  \alpha \right) \left(  \frac14 H_o^2  r_o^2 \right) } \right]^\frac12   \mh (\So) .
		\end{align*}
		Hence, Theorem \ref{thm-extension}  implies  Theorem \ref{thm-intro}. 
	\end{remark}
	
	\begin{remark}
		In Theorem \ref{thm-extension}, if $g$ is close to  a fixed round metric $g_*$ in  $C^{2, \eta}$-norm for some $\eta \in (0,1)$, then, by the proof of  \cite[Proposition 4.1]{M-X},  one can find a particular path $\{ g(t) \}_{ 0 \le t \le 1}$ such that the associated constants  $\alpha $ and $\beta$ satisfy 
		$$ \alpha \le C || g - g_* ||^2_{C^{0, \eta} (\Sigma)  } \ \ \mathrm{and} \ \ \beta \ge 1 - C || g - g_* ||_{C^{2, \eta} (\Sigma) } $$
		for some constant $ C> 0$ independent of $g$. 
		As a result,  $ \alpha \to 0 $ and $ \beta \to 1 $, as $ g $ tends to $ g_*$ in $ C^{2,\eta} (\Sigma)$.
	\end{remark}
	
	\begin{remark}
		Varying $m$ subject to \eqref{eq-thm-choice-m}, Theorem \ref{thm-extension}  gives different 
		estimates $m_{*}$  for the Bartnik mass $\mb (\S)$. 
		We refer readers to  Appendix A of \cite{M-X}  for this optimality analysis  given in a different context. 
	\end{remark}
	
	\vh

\end{document}